\documentclass{amsart}
   \usepackage{amsmath, amssymb,tikz}
\usepackage{hyperref}
\usepackage[margin=1.5cm]{geometry}
\usetikzlibrary{arrows,shapes,positioning,decorations.markings}
\usepackage{tikz-cd}
\usetikzlibrary{cd}

\DeclareMathOperator{\soc}{soc}

\newtheorem{thm}{Theorem}[section]

\newtheorem{lemma}[thm]{Lemma}
\newtheorem{question}[thm]{Question}

\numberwithin{equation}{section}

\begin{document}
	\bibliographystyle{amsplain}

	\title[Maximal blocks]{A polynomial bound for the number of maximal systems of imprimitivity of a finite transitive permutation group}

\author[A. Lucchini]{Andrea Lucchini}
\address{Andrea Lucchini, Dipartimento di Matematica \lq\lq Tullio Levi-Civita\rq\rq,\newline
 University of Padova, Via Trieste 53, 35121 Padova, Italy} 
\email{lucchini@math.unipd.it}
         
\author[M. Moscatiello]{Mariapia Moscatiello}
\address{Mariapia Moscatiello, Dipartimento di Matematica \lq\lq Tullio Levi-Civita\rq\rq,\newline
 University of Padova, Via Trieste 53, 35121 Padova, Italy} 
\email{mariapia.moscatiello@math.unipd.it}

\author[P. Spiga]{Pablo Spiga}
\address{Pablo Spiga, Dipartimento di Matematica Pura e Applicata,\newline
 University of Milano-Bicocca, Via Cozzi 55, 20126 Milano, Italy} 
\email{pablo.spiga@unimib.it}
	\maketitle

        \begin{abstract}We show that, there exists a constant $a$ such that,  for every subgroup $H$ of a finite group $G$, the number of maximal subgroups of $G$ containing $H$ is bounded above by $a|G:H|^{3/2}$. In particular, a transitive permutation group of degree $n$ has at most $an^{3/2}$ maximal systems of imprimitivity. When $G$ is soluble, generalizing a classic result of Tim Wall, we prove a much stroger bound, that is, 
the number of maximal subgroups of $G$ containing $H$ is at most $|G:H|-1$.
          \end{abstract}
\section{Introduction}
Tim Wall in 1961~\cite{Wall} has conjectured that the number of maximal subgroups of a finite group $G$ is less than the group order $|G|$. Wall himself proved the conjecture under the additional hypothesis that $G$ is soluble. The first remarkable progress towards a good understanding of Wall's conjecture is due to Liebeck, Pyber and Shalev~\cite{LPS}; they proved that all, but (possibly) finitely many, simple groups satisfy Wall's conjecture. Actually, Liebeck, Pyber and Shalev prove~\cite[Theorem~$1.3$]{LPS} a polynomial version of Wall's conjecture: there exists an absolute constant $c$ such that, every finite group $G$ has at most $c|G|^{3/2}$ maximal subgroups. Based on the conjecture of Guralnick on the dimension of certain first cohomology groups~\cite{Guralnick} and on some  computer computations of Frank L\"{u}beck, Wall's conjecture was disproved in 2012 by the participants of an AIM workshop, see~\cite{AIM}.

The question of Wall can be generalised in the context of finite permutation groups and this was done by Peter Cameron, see~\cite{Cameron}. (See~\cite{Cameron} also for the motivation for this question.)
\begin{question}[Cameron~\cite{Cameron}]\label{q:1}Is the number of maximal blocks of imprimitivity through a point for a transitive group $G$ of degree $n$ bounded above by a polynomial of degree n? Find the best bound! 
\end{question} 
To see that this question extends naturally the question of Wall we fix some notation. Given a finite group $G$ and a subgroup $H$ of $G$, we denote by $$\max(H,G):=|\{M\mid M \textrm{ maximal subgroup of }G \textrm{ with }H\le M\}|,$$ the number of maximal subgroups of $G$ containing $H$.  Now, if $\Omega$ is the domain of a transitive permutation group $G$ and $\omega\in \Omega$, then there exists a one-to-one correspondence between the maximal systems of imprimitivity of $G$ and the maximal subgroups of $G$ containing the point stabiliser $G_\omega$ and hence Question~\ref{q:1} asks for a polynomial upper bound  for $\max(G_\omega,G)$ as a function of $n=|G:G_\omega|$. When $n=|G|$, that is, $G$ acts regularly on itself, the question of Cameron reduces to the question of Wall and~\cite[Theorem~$1.3$]{LPS} yields a positive solution in this case, with exponent $3/2$.  

The main result of this paper is a positive solution to Question~\ref{q:1}.
\begin{thm}\label{maing}
There exists a constant $a$ such that, for every  finite group $G$ and for every subgroup $H$ of $G$, we have $\max(H,G)\le a|G:H|^{3/2}$.
In particular, a transitive permutation group of degree $n$ has at most $an^{3/2}$ maximal systems of imprimitivity.
\end{thm}
In the case of soluble groups we actually obtain a much tighter bound, which extends the result of Wall~\cite[(8.6), page~$58$]{Wall}  for soluble groups on his own conjecture.
\begin{thm}\label{mains}
If $G$ is a finite soluble group and $H$ is a proper subgroup of $G$, then $\max(H,G)\le |G:H|-1$.
In particular, a soluble transitive permutation group of degree $n\ge 2$ has at most $n-1$ maximal systems of imprimitivity.
\end{thm}

\section{Preliminaries}

We start by reviewing some basic results on  $G$-groups, on monolithic primitive groups and on crowns tailored to our proof of Theorem~\ref{maing}.  For the first part we follow~\cite{crowns}, for the second part we follow~\cite{paz} and for the third part we follow~\cite[Chapter~1]{classes} and~\cite{crowns}. This section will also help for setting some notation. All groups in this paper are finite.

\subsection{Monolithic primitive groups and crown-based power}Recall that an \emph{abstract} group $L$ is said to be \emph{primitive} if it has a maximal subgroup with trivial core. Incidentally, given a group $G$ and a subgroup $M$ be denote by $$\mathrm{core}_G(M):=\bigcap_{g\in G}M^g$$ the core of $M$ in $G$. The \emph{socle} $\mathrm{soc}(L)$ of a primitive group $L$ is either a minimal normal subgroup, or the direct product of two non-abelian minimal normal subgroups. A primitive group $L$ is said to be \emph{monolithic} if the first case occurs, that is, $\mathrm{soc}(L)$ is a minimal normal subgroup of $L$ and hence (necessarily) $L$ has a unique minimal normal subgroup.

Let $L$ be a monolithic primitive group and let $A:=\mathrm{soc}(L)$.  For each positive integer $k$,
let $L^k$ be the $k$-fold direct
product of $L$. The \emph{crown-based power} of $L$ of size  $k$ is the subgroup $L_k$ of $L^k$ defined by
$$L_k:=\{(l_1, \ldots , l_k) \in L^k  \mid l_1 \equiv \cdots \equiv l_k \pmod A \}.$$
Equivalently, if we denote by $\mathrm{diag}(L^k)$ the diagonal subgroup of $L^k$, then $L_k=A^k \mathrm{diag} (L^k)$.

\smallskip

For the proof of the next lemma we need some basic terminology, which we borrow from~\cite[Section~4.3 and~4.4]{praeger}. Let $\kappa$ be a positive integer and let $A$ be a direct product $S_1\times \cdots \times S_\kappa$, where the $S_i$s are pair-wise isomorphic non-abelian simple groups. We denote by $\pi_i:A\to S_i$ the natural projection onto $S_i$. A subgroup $X$ of $A$ is said to be a \emph{strip}, if $X\ne 1$ and, for each $i\in \{1,\ldots,\kappa\}$, either $X\cap \mathrm{Ker}(\pi_i)=1$ or $\pi_i(X)=1$. The support of the strip $X$ is the set $\{i\in \{1,\ldots,\kappa\}\mid \pi_i(X)\ne 1\}$. The strip $X$ is said to be \emph{full} if $\pi_i(X)=S_i$, for all $i$ in the support of $X$. Two strips $X$ and $Y$ are \emph{disjoint} if their supports are disjoint. A subgroup $X$ of $A$ is said to be a \emph{subdirect} subgroup if, for each $i\in \{1,\ldots,\kappa\}$, $\pi_i(X)=S_i$.

Scott's lemma (see for instance~\cite[Theorem~4.16]{praeger}) shows (among other things) that if $X$ is a subdirect subgroup of $A$, then $X$ is a direct product of pairwise disjoint full strips of $A$.

\begin{lemma}\label{lemma:new}
Let $L_{k'}$ be a crown-based power of $L$ of size $k'$ having non-abelian socle $N^{k'}$ and let $H'$ be a core-free subgroup of $L_{k'}$ contained in $N^{k'}$. Then $|N^{k'}:H'|\ge 5^{k'}$. 
\end{lemma}
\begin{proof}
We argue by induction on $k'$. If $k'=1$, then the result is clear because $N^{k'}=N$ has no proper subgroups having index less then $5$. Suppose that $k'\geq 2$ and write $N:=N_1\times \cdots\times N_{k'}$, where $N_1,\ldots,N_{k'}$ are the minimal normal subgroups of $L_{k'}$ contained in $N^{k'}$. For each $i\in \{1,\ldots,k'\}$, we denote by $\pi_i:N^{k'}\to N_i$ the natural projection onto $N_i$. 

Suppose that there exists $i\in \{1,\ldots,k'\}$ with $\pi_i(H')<N_i$. Then, $N_iH'/N_i$ is a core-free subgroup of $L_{k'}/N_i\cong L_{k'-1}$ and is contained in $N^{k'}/N_i$. Therefore, by induction, $|N^{k'}:H'N_i|=|N^{k'}/N_i:H'N_i/N_i|\ge 5^{k'-1}$. Furthermore, $|H'N_i:H'|=|N_i:H'\cap N_i|\ge 5$ because $N_i$ has no proper subgroups having index less then $5$. Therefore, $|N^{k'}:H'|\ge 5^{k'}$.

Suppose that, for every $i\in \{1,\ldots,k'\}$, $\pi_i(H')=N_i$. Since $N$ is non-abelian, we may write $N_i=S_{i,1}\times \cdots \times S_{i,\ell}$, for some pair-wise isomorphic non-abelian simple groups $S_{i,j}$ of cardinality $s$. For each $i\in \{1,\ldots,k'\}$ and $j\in \{1,\ldots,\ell\}$, we denote by $\pi_{i,j}:N^{k'}\to S_{i,j}$ the natural projection onto $S_{i,j}$. As $\pi_i(H')=N_i$, we deduce $\pi_{i,j}(H')=S_{i,j}$, for every $i\in \{1,\ldots,k'\}$ and $j\in \{1,\ldots,\ell\}$.  In particular, $H'$ is a subdirect subgroup of $S_{1,1}\times \cdots \times S_{k',\ell}$ and hence (by Scott's lemma) $H'$ is a direct product of pair-wise disjoint full strips. Since no  $N_i$ is contained in $H'$, there exist two distinct indices $i_1,i_2\in\{1,\ldots,k'\}$ and $j_1,j_2\in \{1,\ldots,\ell\}$ such that $(i_1,j_1)$ and $(i_2,j_2)$ are involved in the same full strip of $H'$. If we now consider the projection $\pi_{i_1,i_2}:N^{k'}\to N_{i_1}\times N_{i_2}$, we obtain $|N_{i_1}\times N_{i_2}:\pi_{i_1,i_2}(H')|\ge s\ge 60\ge 5^2$. The inductive hypothesis applied to $\mathrm{Ker}(\pi_{i_1,i_2})\cap H'$ yields $|\mathrm{Ker}(\pi_{i_1,i_2}):\mathrm{Ker}(\pi_{i_1,i_2})\cap H'|\ge 5^{k'-2}$ and hence $|N^{k'}:H'|\ge 5^{k'}$.
\end{proof}
In the proof of Theorem~\ref{maing} and~\ref{mains}, we use without mention the following basic fact.
\begin{lemma}\label{lemma:remark}
Let $M$ be a normal subgroup of a crown-based power $L_k$ with socle $N^k$. Then either $M\le N^k$ or $N^k\le M$.
\end{lemma}
\begin{proof}
For each $i\in \{1,\ldots,k\}$, we write $N_i:=\{(n_1,\ldots,n_k)\in N^k\mid n_j=1,\, \forall j\in \{1,\ldots,k\}\setminus \{i\}\}$. In particular, $N=N_1\times \cdots \times N_k$.

Let $M$ be a normal subgroup of the crown based power $L_k$ with socle $N^k$ and with $M\nleq N^k$.  Let $m\in M\setminus N^k$. For each $i\in \{1,\ldots,k\}$, since $M$ does not centralize $N_i$, we deduce $1\ne [M,N_i]\le M\cap N_i$. As $N_i$ is one of the minimal normal subgroups of $L_k$, we must have $N_i\le M$. Therefore, $N^k=N_1\times \cdots \times N_k\le M$.  
\end{proof}

\subsection{Basic facts on $G$-groups}Given a group $G$, a $G$-\emph{group} $A$ is a group $A$ together with a group homomorphism $\theta:G\to\mathrm{Aut}(A)$. (For simplicity, we write $a^g$ for the image of $a\in A$ under the automorphism $\theta(g)$.) Given a $G$-group $A$, we have the corresponding \emph{semi-direct product} $A\rtimes_\theta G$ (or simply $A\rtimes G$ when $\theta$ is clear from the context), where the multiplication is given by $$g_1a_1\cdot g_2a_2=g_1g_2 a_1^{g_2}a_2,$$ for every $a_1,a_2\in A$ and for every $g_1,g_2\in G$.
A $G$-group $A$ is said to be \emph{irreducible} if $G$ leaves invariant no non-identity proper normal subgroup of $A$.

Two $G$-groups $A$ and $B$ are said to be $G$-\emph{isomorphic} (and we write $A\cong_G B$), if there exists an isomorphism $\varphi:A\to B$ such that  
$$(a^{g})^{\varphi}=(a^\varphi)^g,$$
for every $a\in A$ and for every $g\in G$. Similarly, we say that $A$ and $B$ are $G$-\emph{equivalent} (and we write $A \sim_G B$), if there exist two isomorphisms $\varphi:A\to B$ and $\Phi: A\rtimes G \rightarrow B\rtimes G$ such that the following diagram commutes. 

\smallskip

\begin{tikzpicture}
\begin{tikzcd}
1\arrow[r,hook]&A\arrow[r,hook]\arrow[d,"\varphi"]&A\rtimes G\arrow[r,two heads]\arrow[d,"\Phi"]&G\arrow[r,two heads]\arrow[d,equal]&1\\
1\arrow[r,hook]&B\arrow[r,hook]&B\rtimes G\arrow[r,two heads]&G\arrow[r,two heads]&1
\end{tikzcd}
\end{tikzpicture}

Being ``$G$-equivalent" is an equivalence relation among $G$-groups coarser than the ``$G$-isomorphic" equivalence relation, that is,  two $G$-isomorphic $G$-groups are necessarily $G$-equivalent. The converse is not necessarily true: for instance, if $A$ and $B$ are two isomorphic non-abelian simple groups and $G:=A	\times B$ acts on $A$ and on $B$ by conjugation, then $A\ncong_GB$ and $A\sim_G B$. However, when $A$ and $B$ are abelian, the converse is true, that is, if $A$ and $B$ are abelian, then $A\sim_GB$ if and only if $A\cong_GB$, see~\cite[page~$178$]{paz}.

Let $G$ be a group and let $A:=X/Y$ be a chief factor of $G$, where $X$ and $Y$ are normal subgroups of $G$. Clearly, the action by conjugation of $G$ endows $A$ of the structure of $G$-group and, in fact, $A$ is an irreducible $G$-group. On the set of chief factors,  the $G$-equivalence relation is easily described. Indeed,  it is proved in~\cite[Proposition~$1.4$]{paz} that two  chief factors $A$ and $B$ of $G$ are  $G$-equivalent if and only if  either 
\begin{itemize}
\item $A$ and $B$ are  $G$-isomorphic, or
\item there exists a maximal subgroup $M$ of $G$ such that $G/\mathrm{core}_G(M)$ has two minimal normal subgroups $N_1$ and $N_2$
$G$-isomorphic to $A$ and $B$ respectively. 
\end{itemize}
(The example in the previous paragraph witnesses that the second possibility does arise.) From this, it follows that, for every monolithic primitive group $L$ and for every $k\in\mathbb{N}$, the  minimal normal subgroups of the crown-based power $L_k$ are all $L_k$-equivalent.

\subsection{Crowns of a finite group}\label{sec:crowns}
Let $X$ and $Y$ be normal subgroups of $G$ with $A=X/Y$ a chief factor of $G$. A \emph{complement} $U$ to $A$ in $G$ is a subgroup $U $ of $G$ such that $$G=UX\, \textrm{ and }\,Y=U \cap X.$$ We say that   $A=X/Y$ is a \emph{Frattini chief factor} if  $X/Y$ is contained in the Frattini subgroup of $G/Y$; this is equivalent to saying that $A$ is abelian and there is no complement to $A$ in $G$.
The  number $\delta_G(A)$  of non-Frattini chief factors $G$-equivalent to $A$   in any chief series of $G$  does not depend on the series and hence $\delta_G(A)$ is a well-defined integer depending only on the chief factor $A$. 

We denote by  $L_A$  the  \emph{monolithic primitive group  associated} to $A$,
that is,
$$L_{A}:=
\begin{cases}
A\rtimes (G/C_G(A)) & \text{ if $A$ is abelian}, \\
G/C_G(A)& \text{ otherwise}.
\end{cases}
$$

If $A$ is a non-Frattini chief factor of $G$, then $L_A$ is a homomorphic image of $G$.
More precisely,  there exists
a normal subgroup $N$ of $G$ such that $$G/N \cong L_A\,\textrm{ and }\,\soc(G/N)\sim_GA.$$ Consider now the collection $\mathcal{N}_A$ of all  normal subgroups $N$ of $G$ with $G/N \cong L_A$ and $\soc(G/N)\sim_G A$:
the intersection $$R_G(A):=\bigcap_{N\in\mathcal{N}_A}N$$ has the property that  $G/R_G(A)$ is isomorphic to the crown-based  power $(L_A)_{\delta_G(A)}$, that is, $G/R_G(A)\cong (L_A)_{\delta_G(A)}$.

The socle $I_G(A)/R_G(A)$ of $G/R_G(A)$ is called the $A$-\emph{crown} of $G$ and it is  a direct product of $\delta_G(A)$ minimal normal subgroups all $G$-equivalent to $A$.

We conclude this preliminary section with two technical lemmas and one of the main results from~\cite{LPS}.
\begin{lemma}{\cite[Lemma 1.3.6]{classes}}\label{sotto}
	Let $G$ be a finite  group with trivial Frattini subgroup. There exists
	a chief factor  $A$ of $G$ and a non-identity normal subgroup $D$ of $G$ with $I_G(A)=R_G(A)\times D.$
\end{lemma}

\begin{lemma}{\cite[Proposition 11]{crowns}}\label{sotto2}
Let $G$ be a finite  group with trivial Frattini subgroup, let $I_G(A), R_G(A)$ and $D$ be as in the statement of Lemma~$\ref{sotto}$ and let $K$ be a subgroup of $G$. If $G=KD=KR_G(A),$ then $G=K.$
\end{lemma}

\begin{thm}\label{pyber}{\cite[Theorem~$1.4$]{LPS}}There exists a constant $c$ such that every finite group has at most $cn^{3/2}$ core-free maximal subgroups of index $n$.
\end{thm}
Theorem~\ref{pyber} is an improvement of~\cite[Corollary~2]{Pyber}. We warn the reader that the statement of Theorem~\ref{pyber} is slightly different from that of Theorem~$1.4$ in~\cite{LPS}: to get Theorem~\ref{pyber} one should take into account Theorem~$1.4$ in~\cite{LPS} and the remark following its statement.

\section{Proof of Theorems~$\ref{maing}$ and~$\ref{mains}$}
In this section we prove Theorems~$\ref{maing}$ and~\ref{mains}. Our proofs are inspired from some ideas developed in~\cite{ik}. Moreover, our proofs have some similarities and hence we start by deducing some general facts holding for both. 

We start by defining the universal constant $a$. Observe that the series $\sum_{u=1}^\infty u^{-3/2}$ converges. We write  
$$a':=\sum_{u=1}^\infty \frac{1}{u^{3/2}}.$$
Let $c$ be the universal constant arising from  Theorem~\ref{pyber}. We define
$$a:=\frac{11ca'}{1-1/2^{3/2}}.$$

Recall that $\max(H,G)$ is the number of  maximal subgroups of $G$ containing $H$. For the proof of Theorems~$\ref{maing}$ and~\ref{mains} we argue by induction on $|G:H|+|G|$. The case $|G:H|=1$ for the proof of Theorem~\ref{maing} is clear because $\max(H,G)=0$. Similarly,  the case that $H$ is maximal in $G$ for the proof of Theorem~\ref{mains} is clear because $\max(H,G)=1$. In particular, for the proof of Theorem~\ref{maing}, we suppose $|G:H|>1$  and, for the proof of Theorem~\ref{mains}, we suppose that $H$ is not maximal in $G$. 

Consider $$\tilde{H}:=\bigcap_{\substack{H\le M<G\\M \textrm{ max. in }G}}M.$$ 
Observe that $\max(H,G)=\max(\tilde{H},G)$. In particular, when $H<\tilde{H}$, we have $|G:\tilde{H}|<|G:H|$ and hence, by induction, we have $\max(H,G)=\max(\tilde{H},G)\le a|G:\tilde{H}|^{3/2}<a|G:H|^{3/2}$; moreover, when $G$ is soluble, we have $\max(H,G)=\max(\tilde{H},G)\le |G:\tilde{H}|-1<|G:H|-1$. Therefore, we may suppose  $H=\tilde{H}$, that is, 
\begin{equation}\label{eq:2}
H\textrm{ is an intersection of maximal subgroups of }G.
\end{equation}

Suppose that $H$ contains a non-identity normal subgroup $N$ of $G$. Since $\max(H,G)=\max(H/N,G/N)$ and $|G/N|<|G|$, by induction, we have $\max(H,G)=\max(H/N,G/N)\le a|G/N:H/N|^{3/2}=a|G:H|^{3/2}$; moreover, when $G$ is soluble, we have $\max(H,G)=\max(H/N,G/N)\le |G/N:G/N|-1=|G:H|-1$. Therefore, we may suppose
\begin{equation}\label{eq:3}
\mathrm{core}_G(H)=1.
\end{equation}

Let $F$ be the Frattini subgroup of $G$. From~\eqref{eq:2}, we have $F\le H$ and hence, from~\eqref{eq:3}, $F=1$. 
In particular, we may now apply Lemma~\ref{sotto} to the group $G$.

Choose $I$, $R$ and $D$ as in Lemma \ref{sotto}. 
From~\eqref{eq:2}, we may write
$$H=X_1\cap\cdots\cap X_{\rho}\cap Y_1\cap\cdots\cap Y_{\sigma},$$
where  $X_1,\dots,X_\rho$ are the maximal subgroups of $G$ not containing $D$ 
and $Y_1,\dots,Y_\sigma$ are the  maximal subgroups of $G$ containing $D.$
	We define 
$$X:=X_1\cap\cdots\cap X_{\rho}\, \text{ and }\, Y:=Y_1\cap\cdots\cap Y_{\sigma}.$$
Thus $H=X\cap Y$.

For every $i\in \{1,\ldots,\rho\}$, since $D\nleq X_i$, we have $G=DX_i$ and hence Lemma~\ref{sotto2} (applied with $K:=X_i$) yields $R\le X_i$. In particular,
\begin{equation}\label{eq:5}
R\le X.
\end{equation}

Since $R=R_G(A)$ for some chief factor $A$ of $G$, Section~\ref{sec:crowns} yields $$G/R\cong L_k,$$ 
for some monolithic primitive group $L$ and for some positive integer $k$. We let $N$ denote the minimal normal subgroup (a.k.a. the socle) of $L$. From the definition of $I$ and $R$, we have $I/R=\soc(G/R)\cong \soc(L_k)=N^k$.   Finally, let $T:=X\cap I$. In particular, $$\frac{T}{R}=\frac{X}{R}\cap \frac{I}{R}.$$

We have
$$H\cap D=(X\cap Y)\cap D=X\cap (Y\cap D)=X\cap D=X\cap (I\cap D)=(X\cap I)\cap D =T\cap D.$$
It follows
$$|G:HD|=\frac{|G:H|}{|HD:H|}=\frac{|G:H|}{|D:H\cap D|}=\frac{|G:H|}{|D:T\cap D|}.$$
If $D\le T$, then $D\le X$ and hence $D\le X\cap Y=H $ because $D\le Y$; however this is a contradiction because $D\ne 1$ and hence, from~\eqref{eq:3}, $D\not\leq H.$ Therefore $D\nleq T$ and $|D:T\cap D|>1$.
 
Applying our inductive hypothesis, we obtain 
\begin{equation}\label{eq:sigma}\sigma=\max(HD/D,G/D)\leq a|G/D:HD/D|^{3/2}=a|G:HD|^{3/2}=a\left(\frac{|G:H|}{|D:D\cap T|}\right)^{3/2}\le \frac{a}{2^{3/2}}|G:H|^{3/2};
\end{equation}
moreover,  when $G$ is soluble and $HD$ is a proper subgroup of $G$, we obtain 
\begin{equation}\label{eq:sigmas}
\sigma=\max(HD/D,G/D)\leq |G/D:HD/D|-1=|G:HD|-1=\frac{|G:H|}{|D:D\cap T|}-1\le \frac{|G:H|}{2}-1.
\end{equation}
(Observe that, when $G$ is soluble and $G=HD$, we have $\sigma=0$ and hence the inequality $\sigma\le |G:H|/2-1$ is valid also in this degenerate case.)

From~\eqref{eq:5}, we deduce $\rho\le\max(HR,G)$. If $R\nleq H$, then $|G:HR|<|G:H|$ and hence, applying our inductive hypothesis, we obtain
\begin{equation}\label{eq:newg}\rho\le\max(HR,G)\le a|G:HR|^{3/2}= a\left(\frac{|G:H|}{|HR:H|}\right)^{3/2}\le\frac{a}{2^{3/2}}|G:H|^{3/2};\end{equation}
moreover,  when $G$ is soluble and $HR$ is a proper subgroup of $G$, we obtain 
\begin{equation}\label{eq:news}\rho\le\max(HR,G)\le |G:HR|-1= \frac{|G:H|}{|HR:H|}-1\le\frac{|G:H|}{2}-1.\end{equation}
(As above, when $G$ is soluble and $G=HR$, we have $\rho=0$ and hence the inequality $\rho\le |G:H|/2-1$ is valid also in this degenerate case.)

Now, from~\eqref{eq:sigma} and~\eqref{eq:newg}, we have $$\max(H,G)=\sigma+\rho\le \frac{2a}{2^{3/2}}\cdot|G:H|^{3/2}<a|G:H|^{3/2};$$
similarly, when $G$ is soluble, from~\eqref{eq:sigmas} and~\eqref{eq:news}, we have
$$\max(H,G)=\sigma+\rho\le \frac{|G:H|}{2}-1+\frac{|G:H|}{2}-1<|G:H|-1.$$
 In particular, for the rest of the proof, we may assume that $R\le H$. Now,~\eqref{eq:3} yields $R=1$ and hence $G\cong L_k$ and $D=I$. Therefore, we may  identify $G$ with $L_k$ and $D$ with $N^k$.

Set $$\mathcal{C}:=\{\mathrm{core}_G(X_i)\mid i\in \{1,\ldots,\rho\}\}$$
and, for every $C\in\mathcal{C}$, set 
$$\mathcal{M}_C:=\{X_i\mid i\in \{1,\ldots,\rho\}, C=\mathrm{core}_G(X_i)\}.$$
 
For the rest of our argument for proving Theorems~\ref{maing} and~\ref{mains}, we prefer to keep the proofs separate.

\begin{proof}[Proof of Theorem~$\ref{maing}$]

\noindent\textsc{Case 1: }Suppose that $N$ is non-abelian.

\smallskip 

\noindent Since $N$ is non-abelian, the group $G=L_k$ has exactly $k$ minimal normal subgroups.  We denote by $N_1,\ldots,N_k$ the minimal normal subgroups of $G$. In particular, $I=N^k=N_1\times N_2\times\cdots \times N_k$.

 We claim that, for every $i\in \{1,\ldots,\rho\}$, there exist $x,y\in \{1,\ldots,k\}$ such that $N_\ell\le X_i$, for every $\ell\in \{1,\ldots,k\}\setminus\{x,y\}$, that is, $X_i$ contains all but possibly at most two minimal normal subgroups of $G$.

We argue by induction on $k$. The statement is clearly true when $k\leq 2$. Suppose then $k\ge 3$ and let $C:=\mathrm{core}_{G}(X_i)$. If $C=1$, then $X_i$ is a maximal core-free subgroup of $G$ and hence the action of $G$ on the right cosets of $X_i$ gives rise to a faithful primitive permutation representation. Since a primitive permutation group has at most two minimal normal subgroups~\cite[Theorem~$4.4$]{CameronBook} and since $G$ has exactly $k$ minimal normal subgroups, we deduce that $k\le 2$, which is a contradiction. Therefore $C\neq 1$.

Since $N_1,\ldots,N_k$ are the minimal normal subgroups of $L_k$, we deduce that there exists $\ell\in \{1,\ldots,k\}$ with $N_\ell\le C$. Now, the proof of the claim follows applying the  inductive hypothesis to $G/N_\ell\cong L_{k-1}$ and to its maximal subgroup $X_i/N_\ell$.

The previous claim shows that, for every $C\in\mathcal{C}$, $C$ contains all but possibly at most two minimal normal subgroups of $N^k=I$. Therefore, $$|\mathcal{C}|\le k^2.$$

Let $C\in \mathcal{C}$ and let $M\in \mathcal{M}_C$. The reader might find useful to see Figure~\ref{fig:1}, where we have drawn a fragment of the subgroup lattice of $G$ relevant to our argument.

\begin{figure}[!h]
\begin{tikzpicture}[node distance   = 1cm ]
        \node(A){$G$};
\node[below=of A](B){$HI$};
\node[below=of B](C){$I$};
\node[right=of B](AA){$\,\,\,\,\,\,\,M\,\,\,\,\,\,\,$};
\node[right=of C](D){$H(I\cap M)$};
\node[below=of D](E){$I\cap M$};
\node[right=of E](F){$H$};
\node[below=of F](G){$I\cap H$};
\draw(A)--node[left]{$y$}(B);
\draw(B)--node[left]{$x$}(C);
\draw(A)--node[above]{$z$}(AA);
\draw(AA)--node[left]{$y$}(D);
\draw(B)--node[above]{$z$}(D);
\draw(E)--node[above]{$z$}(C);
\draw(E)--node[left]{$x$}(D);
\draw(D)--node[above]{$t$}(F);
\draw(E)--node[above]{$t$}(G);
\draw(F)-- node[left]{$x$}(G);
\end{tikzpicture}
\caption{Subgroup lattice for $G$}\label{fig:1}
\end{figure}
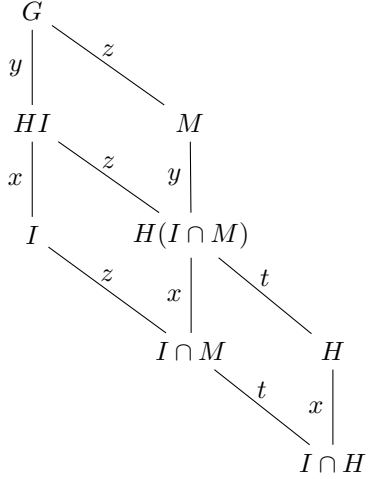
Let $k'$ be the number of minimal normal subgroups of $G$  contained in $M$. In particular, $I\cap M\cong N^{k'}$. Observe that $I\cap H$ is contained in $I\cap M$ and is core-free in $G$. Applying Lemma~\ref{lemma:new} (with $H'$ replaced by $I\cap H$ in a crowned-based group isomorphic to $L_{k'}$), we get $|I\cap M:I\cap H|\ge 5^{k'}$. As $k'\ge k-2$, we deduce $t\ge 5^{k-2}$. 

Now, $M/C$ is a core-free maximal subgroup of $G/C$. From Theorem~\ref{pyber}, when $C=\mathrm{core}_G(M)$ and $z=|G:C|$ are fixed,  we have at most $cz^{3/2}$ choices for $M$. As $t\ge 5^{k-2}$, we have $z\le |G:H|/5^{k-2}$. 
Thus
\begin{align*}
\rho&=\sum_{C\in\mathcal{C}}|\mathcal{M}_C|\le 
\sum_{C\in\mathcal{C}}\sum_{\substack{z\mid |G:H|\\z\le |G:H|/5^{k-2}}}cz^{3/2}\le ck^2\sum_{\substack{z\mid |G:H|\\z\le |G:H|/5^{k-2}}}z^{3/2}=ck^2\left(\frac{|G:H|}{5^{k-2}}\right)^{3/2}\sum_{\substack{z\mid |G:H|\\z\le |G:H|/5^{k-2}}} \left(\frac{5^{k-2}z}{|G:H|}\right)^{3/2}.
\end{align*}
 Therefore, 
\begin{align*}
\sum_{\substack{z\mid |G:H|\\z\le |G:H|/5^{k-2}}} \left(\frac{5^{k-2}z}{|G:H|}\right)^{3/2}\le \sum_{u=1}^\infty\frac{1}{u^{3/2}}=a'.
\end{align*}
Finally, it is easy to verify that, for every $k$, $k^2/5^{3(k-2)/2}\le 11$. Summing up,
\begin{equation}\label{ugly}
\rho\le 11ca'|G:H|^{3/2}.
\end{equation}
From~\eqref{eq:sigma},~\eqref{ugly} and from the definition of $a$, we have
$$\max(H,G)=\sigma+\rho\le \frac{a}{2^{3/2}}|G:H|^{3/2}+11ca'|G:H|^{3/2}=a|G:H|^{3/2}.$$

\smallskip

\noindent\textsc{Case 2: }Suppose that $N$ is abelian.

\smallskip

\noindent As $N$ is abelian, the action of $L$ by conjugation on $N$ endows $N$ of the structure of an $L$-module. Since $L$ is primitive, $N$ is irreducible. Set $q:=|\mathrm{End}_L(N)|$. Now, $N$ is a vector space over the finite field $\mathbb{F}_q$ with $q$ elements, and hence $|N|=q^{k'}$, for some positive integer $k'$.

  Let $C\in \mathcal{C}$ and let $M\in \mathcal{M}_C$. From Lemma~\ref{lemma:remark}, $C\le I$. Now, the action of $G/C$ on the right cosets of $M/C$ is a primitive permutation group with point stabilizer $M/C$. Observe that in this primitive action, $I/C$  is the socle of $G/C$. In particular, $G/C$ acts irreducibly as a linear group on $I/C$ and hence $C$ is a maximal $L$-submodule of $I$. Since $I$ is the direct sum of $k$ pairwise isomorphic irreducible $L$-modules, we deduce that we have at most $(q^k-1)/(q-1)$ choices for $C$. Moreover, $|G:M|=|G/C:M/C|=|N|=q^{k'}$. From Theorem~\ref{pyber}, when $C$ is fixed, we have at most $c|G:M|^{3/2}=c(q^{k'})^{3/2}$ choices for $M\in\mathcal{M}_C$. This yields 
\begin{equation}\label{eq:refine}\rho\le |\mathcal{C}|\cdot\max_{C\in\mathcal{C}}|\mathcal{M}_C|\le \frac{q^k-1}{q-1}\cdot cq^{3k'/2}<cq^{k+3k'/2}.
\end{equation}
As we have observed above, $M\cap I=C$ is an $L$-submodule of $G$. Since an intersection of $L$-submodules is an $L$-submodule, we deduce that 
$$H\cap I=(X_1\cap \cdots \cap X_\rho)\cap I$$
is an $L$-submodule of $I$ and hence $H\cap I\unlhd G$. Since $H$ is core-free in $I$, we deduce $H\cap I=1$ and hence $|I|=|N|^k=q^{kk'}$ divides $|G:H|$. In particular, $|G:H|\le q^{kk'}$.
Therefore, from~\eqref{eq:refine}, we obtain
$$\rho \le c|G:H|^{\frac{k+3k'/2}{kk'}}.$$
When $k\ne 1$ or when $(k,k')\ne (2,1)$, we have $\frac{k+3k'/2}{kk'}\le 3/2$. When $k=1$, by refining~\eqref{eq:refine}, we obtain the sharper bound $\rho \le cq^{3k'/2}\le c|G:H|^{3/2}$. When $(k,k')=(2,1)$, we may refine again~\eqref{eq:refine}: $\rho\le c(q+1)q^{3/2}\le c\cdot 2q\cdot q^{3/2}=2cq^{5/2}\le 2c |G:H|^{5/4}\le 2c|G:H|^{3/2}$. Summing up, in all cases we have
\begin{equation}\label{eq:final}
\rho \le 2c|G:H|^{3/2}.
\end{equation}

From~\eqref{eq:sigma} and~\eqref{eq:final}, we have
$$\max(H,G)=\sigma+\rho\le \frac{a}{2^{3/2}}|G:H|^{3/2}+2c|G:H|^{3/2}<a|G:H|^{3/2}.$$
\end{proof}

The rest of the proof of Theorem~\ref{mains} follows the same idea as in the ``Case 2" above, but taking in account that the whole group $G$ is soluble.
\begin{proof}[Proof of Theorem~$\ref{mains}$]
Since $G=L_k$ and $I=N^k$, we may write $G=I\rtimes K$, where $K$ is a complement of $N$ in $L$. As in the proof of Theorem~\ref{maing} for the case that $N$ is abelian, we have that the action of $L$ by conjugation on $N$ endows $N$ of the structure of an $L$-module. Since $L$ is primitive, $N$ is irreducible. Set $q:=|\mathrm{End}_L(N)|$. Now, $N$ is a vector space over the finite field $\mathbb{F}_q$ with $q$ elements, and hence $|N|=q^{k'}$, for some positive integer $k'$.

Let $C\in \mathcal{C}$ and let $M\in \mathcal{M}_C$. As we have observed above (for the proof of ``Case~2''), $M\cap I=C$ is a maximal $L$-submodule of $G$, $H\cap I=1$ and $|I|=|N|^k=q^{kk'}$ divides $|G:H|$. In particular, $|G:H|=\ell q^{kk'}$, for some positive integer $\ell$.

 Since $G$ is soluble and since $M$ is a maximal subgroup of $G$ supplementing $I$, we have $M=C\rtimes K^x$, for some maximal $L$-submodule $C$ of $I$ and some $x\in I$.
Arguing as in the proof of Theorem~\ref{maing} for the case that $N$ is abelian, we deduce that we have at most $(q^k-1)/(q-1)$ choices for $C$. Moreover, we have at most $|I/C|=|G:M|=|N|=q^{k'}$ choices for $x$.
This yields 
\begin{equation}\label{eq:refinebis}\rho\le \frac{q^k-1}{q-1}q^{k'}.
\end{equation}

Now,~\eqref{eq:sigmas} gives $\sigma\le |G:H|/|D:D\cap T|-1$: recall that $D=I=N^k$ and $D\cap T=D\cap H=I\cap H=1$. Thus $\sigma\le |G:H|/|D|-1=|G:H|/q^{kk'}-1=\ell-1$. Therefore,
\begin{equation}\label{eq:tired}
\max(H,G)=\sigma+\rho\le \ell-1+\frac{q^k-1}{q-1}q^{k'}.
\end{equation}

When $\ell \ge 2$, a computation shows that the right hand side of~\eqref{eq:tired} is less than or equal to $\ell q^{kk'}-1=|G:H|-1$. In particular, we may suppose that $\ell=1$.  In this case, $|G:H|=q^{kk'}=|I|$ and hence $G=IH=I\rtimes H$. Moreover, $\sigma=0$. Since $H$ is not a maximal subgroup of $G$ (recall the base case for our inductive argument), $k\ge 2$.

Assume also $k'=1$. Since $|\mathrm{End}_L(N)|=q=|N|$, we deduce that $L/N$ is isomorphic to a subgroup of the multiplicative group of the field $\mathbb{F}_q$ and hence $|L:N|$ is relatively prime to $q$. Therefore $|G:I|$ is relatively prime to $q$ and hence so is $|H|$. Therefore, replacing $H$ by a suitable $G$-conjugate, we may suppose that $K=H$. Using this information, we may now refine our earlier argument bounding $\rho$. Let $C\in \mathcal{C}$ and let $M\in \mathcal{M}_C$. Since $G=I\rtimes H$ is soluble,  $M$ is a maximal subgroup of $G$ supplementing $I$ and $H\le M$, we have $M=C\rtimes H$, for some maximal $L$-submodule $C$ of $I$. We deduce that we have at most $(q^k-1)/(q-1)$ choices for $C$ and hence we have at most $(q^k-1)/(q-1)$ choices for $M$. This yields
$$\max(H,G)=\sigma+\rho=\rho\le \frac{q^k-1}{q-1}\le q^k-1=|G:H|-1,$$
and the result is proved in this case.

Assume $k'\ge 2$. A computation (using $\ell=1$ and $k,k'\ge 2$) shows that the right hand side of~\eqref{eq:tired} is less than or equal to $q^{kk'}-1=|G:H|-1$. 
\end{proof}

\end{document}